\documentclass[a4paper,12pt]{amsart}

\usepackage{amssymb, amscd}
\usepackage{latexsym,epsfig}
\usepackage[all]{xy}
\usepackage{pst-func}
\usepackage{pst-math}
\usepackage{pst-all}
\usepackage{auto-pst-pdf}
\usepackage{hyperref}
\usepackage{color}
\numberwithin{equation}{section}
\def\today{\ifcase\month\or Jan\or Febr\or  Mar\or  Apr\or May\or Jun\or  Jul\or Aug\or  Sep\or  Oct\or Nov\or  Dec\or\fi \space\number\day, \number\year}

\newcommand{\CC}{\mathbb C}
\newcommand{\DD}{\mathbb D}
\newcommand{\EE}{\mathbb E}
\newcommand{\FF}{\mathbb F}

\newcommand{\PP}{\mathbb P}
\newcommand{\QQ}{\mathbb Q}

\newcommand{\ZZ}{\mathbb Z}

\newcommand\langepijl[1]{\buildrel {#1} \over \longrightarrow}

\numberwithin{equation}{section}

\newtheorem{theorem}{Theorem}[section]
\newtheorem{lemma}[theorem]{Lemma}
\newtheorem{proposition}[theorem]{Proposition}
\newtheorem{conclusion}[theorem]{Conclusion}
\newtheorem{corollary}[theorem]{Corollary}

\newtheorem{definition-lemma}[theorem]{Definition-Lemma}

\theoremstyle{definition}

\theoremstyle{remark}
\newtheorem{remark}[theorem]{Remark}

\begin{document}

\title[Siegel Modular Forms  and Invariant Theory]{Siegel 
Modular Forms of Degree two and Three\\
and Invariant Theory}

\author{Gerard van der Geer}
\address{Korteweg-de Vries Instituut, Universiteit van
Amsterdam, Postbus 94248,
1090 GE  Amsterdam, The Netherlands}
\email{g.b.m.vandergeer@uva.nl}

\begin{abstract} This is a survey based 
on the construction of Siegel modular forms
of degree $2$ and $3$ using invariant theory in joint work with
Fabien Cl\'ery and Carel Faber. 
\end{abstract}

\maketitle
%%%%%%%%%%%%%%%%%%%%%%%%%%%%%%%%%%%%%%%%
\hfill{\it Dedicated to Ciro Ciliberto on the occasion of his 70th birthday}
%\centerline{\today}
%%%%%%
\begin{section}{Introduction}
Modular forms are sections of naturally defined vector bundles on arithmetic
quotients of bounded symmetric domains. Often such quotients can be
interpreted as moduli spaces and sometimes this moduli interpretation
allows a description as a stack quotient under the action of an algebraic
group like ${\rm GL}_n$. In such cases classical invariant theory can be used
for describing modular forms.

In the 1960s Igusa used the close connection between the moduli of 
principally polarized complex abelian surfaces and the moduli of 
algebraic curves of genus two to describe the ring of scalar-valued 
Siegel modular forms of degree two (and level $1$) 
in terms of invariants of the action
of ${\rm GL}_2$ acting on binary sextics, 
see \cite{Igusa1960,Igusa1967}. 
Igusa used theta functions 
and a crucial step in Igusa's approach were Thomae's formulas from the 19th century 
that link theta constants for hyperelliptic curves to the cross ratios of
the branch points of the canonical map of the hyperelliptic curve to
${\PP}^1$. 

In the 1980s Tsuyumine, continuing the work of Igusa, 
used the connection between
the moduli of abelian threefolds and curves of genus $3$ to describe
generators for the ring of scalar-valued Siegel modular forms of degree
$3$ (and level $1$). He used the moduli of hyperelliptic curves
of genus $3$ as an intermediate step and used theta functions 
and the invariant theory
of binary octics as developed by Shioda, see \cite{Tsuyumine,Shioda}. 

The description of the moduli of curves of genus $2$ (resp.\ $3$) in terms of
a stack quotient of ${\rm GL}_2$ acting on binary sextics (resp.\ of 
${\rm GL}_3$ acting on ternary quartics) makes it possible to construct
the modular forms directly from the stack quotient without the recourse
to theta functions or cross ratios. This applies not only to
scalar-valued modular forms, but to vector-valued
modular forms as well. Covariants (or concomitants) yield explicit
modular forms in an efficient way. This in contrast to earlier
and more laborious methods of constructing
vector-valued Siegel modular forms of degree $2$ and $3$ that
use theta functions.

In joint work with Fabien Cl\'ery and Carel Faber 
\cite{CFvdG,CFvdG2,CFvdG3}
we exploited this for the construction of
 Siegel modular forms of degree $2$ and $3$.
In degree $2$ the universal binary sextic,
the most basic covariant, defines a meromorphic Siegel modular form
$\chi_{6,-2}$ of weight $(6,-2)$. Substituting the coordinates
of $\chi_{6,-2}$ in a covariant produces a meromorphic modular form
that becomes holomorphic after multiplication by an appropriate
power of $\chi_{10}$, a cusp form of weight $10$ associated to the
discriminant.
For degree $3$  we can play a similar game, now 
involving the universal ternary quartic and
a meromorphic Teichm\"uller modular form $\chi_{4,0,-1}$ 
of weight $(4,0,-1)$ that 
becomes a holomorphic Siegel modular form $\chi_{4,0,8}$ of weight 
$(4,0,8)$ after multiplication with
$\chi_9$, a Teichm\"uller form of weight $9$ related to the discriminant. 

With this approach it is easy to retrieve Igusa's result on the ring
of scalar modular forms of degree $2$.
Another advantage of this direct approach is that one can treat modular forms
in positive characteristic as well.  
Thus it enabled the determination of the rings 
of scalar valued modular forms of degree $2$ in characteristic $2$ and $3$,
two cases that were unaccounted for so far, see \cite{vdG2020,C-vdG3}.

In this survey we sketch the approach and indicate how one constructs
Siegel modular forms of degree $2$ and $3$.  We show how to derive
the results on the rings of scalar-valued modular forms of degree $2$.
\end{section}
%%%%%%%%%%%%%%%%%%%%%%%%%%%%%%%%%%%%%%%%%%%%%%%%%%
\begin{section}{Siegel Modular Forms}
Classically, Siegel modular forms are described as functions on the
Siegel upper half space. We recall the definition.

For $g\in {\ZZ}_{\geq 0}$ we set $\mathcal{L}={\ZZ}^{2g}$ with generators
$e_1,\ldots,e_g,f_1,\ldots,f_g$ and define a symplectic form 
$\langle \, , \, \rangle$ on $\mathcal{L}$ via $\langle e_i,f_j\rangle =\delta_{ij}$.
The Siegel modular group $\Gamma_g={\rm Aut}(\mathcal{L},\langle \, , \, \rangle)$ 
of degree $g$ is the automorphism group of this symplectic lattice. 
Here we write an
element $\gamma \in \Gamma_g$ as a matrix
$\left( \begin{matrix}a & b \\ c & d \\ \end{matrix} \right)$
of four $g\times g$ blocks
using the basis $e_i$ and $f_i$; we often 
abbreviate this as $\gamma=(a,b;c,d)$.
The group $\Gamma_g$
 acts on the Siegel upper half space 
$$
\mathfrak{H}_g=\{
\tau \in {\rm Mat}(g\times g, {\CC}): \tau^t=\tau, \, {\rm Im}(\tau)>0\}
$$
via $\tau \mapsto \gamma(\tau)  = (a \tau+b)(c \tau+d)^{-1}$. 

A scalar-valued Siegel modular form of weight $k$ and
degree $g>1$ is a holomorphic function $f: \mathfrak{H}_g \to {\CC}$
satisfying $f(\gamma(\tau))=\det(c\tau+d)^k f(\tau)$ for all $\gamma=(a,b;c,d) \in 
\Gamma_g$. If $\rho: {\rm GL}(g) \to {\rm GL}(V)$ 
is a complex representation of ${\rm GL}_g$ then a 
vector-valued Siegel modular form of weight $\rho$ and degree $g>1$ 
is a holomorphic map $f: \mathfrak{H}_g \to V$ satisfying
$$
f(\gamma(\tau))= \rho(c\tau+d) f(\tau) \qquad \text{\rm for all 
$\gamma=(a,b;c,d) \in \Gamma_g$} \eqno(1)
$$  
We may restrict to irreducible representations $\rho$.
For $g=1$ we have to require an additional growth condition 
for $y={\rm Im}(\tau) \to \infty$.

\smallskip

However, for an algebraic geometer modular forms are sections of vector bundles.
Let $\mathcal{A}_g$ be the moduli space of principally polarized abelian
varieties of dimension $g$. This is a Deligne-Mumford stack of relative
dimension $g(g+1)/2$ over ${\ZZ}$. It carries a universal principally polarized
abelian variety $\pi:\mathcal{X}_g \to \mathcal{A}_g$. This provides $\mathcal{A}_g$
with a natural vector bundle ${\EE}={\EE}^{(g)}$, the Hodge bundle, defined as
$$
{\EE}=\pi_{*}(\Omega^1_{\mathcal{X}_g/\mathcal{A}_g})\, .
$$
Starting from ${\EE}$ we can create new vector bundles. Each
irreducible representation $\rho$ of ${\rm GL}_g$ defines a vector bundle
${\EE}_{\rho}$ by applying a Schur functor (or just by applying $\rho$
to the transition functions of ${\EE}$). In particular, we have the
determinant line bundle $L=\det({\EE})$. Scalar-valued modular forms
of weight $k$ are sections of $L^{\otimes k}$ and these form a graded ring.
In fact, for $g\geq 2$ and each commutative ring $F$ we have the ring
$$
R_g(F)=\oplus_k H^0(\mathcal{A}_g \otimes F, L^{\otimes k})\, .
$$

The moduli space $\mathcal{A}_g$ can be compactified. 
There is the Satake compactification, 
in some sense a minimal compactification,
based on the fact that $L$ is an ample line bundle 
on $\mathcal{A}_g$.
This compactification $\mathcal{A}_g^{\ast}$ 
is defined as ${\rm Proj}(R_g)$ and satisfies
the inductive property 
$$
\mathcal{A}_g^{\ast}=\mathcal{A}_g \sqcup \mathcal{A}_{g-1}^{\ast}\, .
$$
Restriction to the `boundary' $\mathcal{A}_{g-1}^{\ast}$ induces a map
called the Siegel operator
$$
\Phi: R_g(F) \to R_{g-1}(F)\, .
$$
We will also use (smooth)
Faltings-Chai type compactifications $\tilde{\mathcal{A}}_g$ 
and over these the Hodge bundle extends (\cite{F-C}). 
We will denote the extension also by ${\EE}$.

For $g>1$ the Koecher principle holds:
sections of ${\EE}_{\rho}$ over $\mathcal{A}_g$ extend to regular sections 
of the extension of ${\EE}_{\rho}$ over $\tilde{\mathcal{A}}_g$,
see \cite[Prop 1.5, p.\ 140]{F-C}. For $g=1$
this does not hold since the boundary in $\mathcal{A}_1^{\ast}$ is a divisor,
and we define modular forms of weight $k$ as sections
of $L^{\otimes k}$ over $\tilde{\mathcal{A}}_1$. If $D$ denote the divisor
added to $\tilde{\mathcal{A}}_g$ to compactify $\mathcal{A}_g$, then 
elements of $H^0(\tilde{\mathcal{A}}_g,{\EE}_{\rho}\otimes \mathcal{O}(-D))$
are called cusp forms.

We will write $M_{\rho}(\Gamma_g)(F)$ for $H^0(\tilde{\mathcal{A}}_g\otimes F, {\EE}_{\rho})$ or simply $M_{\rho}(\Gamma_g)$ when $F$ is clear. The space of cusp forms 
is denoted by $S_{\rho}(\Gamma_g)$. By the Koecher principle the spaces
$M_{\rho}(\Gamma_g)(F)$ and $S_{\rho}(\Gamma_g)(F)$ do not
depend on the choice of a Faltings-Chai compactification.

Over the complex numbers if $\rho: {\rm GL}(g) \to {\rm GL}(V)$ is an 
irreducible representation, elements of 
$H^0(\tilde{\mathcal{A}}_g\otimes {\CC},{\EE}_{\rho})$
correspond to holomorphic functions $f: \mathfrak{H}_g \to V$ satisfying (1).
Such a function allows a Fourier expansion
$$
f(\tau)= \sum_{n\geq 0} a(n) \, q^n\, ,
$$
where the sum is over symmetric $g\times g$ half-integral matrices (meaning
$2n$ is integral and even on the diagonal) which are positive semi-definite, $a(n)\in V$ 
and $q^n$ is shorthand for $e^{2\pi i {\rm Tr}(n \tau)}$.

The definition
$$
R_g(F)= \oplus_k H^0(\tilde{\mathcal{A}}_g\otimes F, L^{\otimes k})
$$
for a commutative ring $F$ allows one speak of modular forms in positive
characteristic by taking $F={\FF}_p$. 
One cannot define such modular forms by Fourier series.

\smallskip
We summarize what is known about the rings $R_g(F)$.
It is a classical result that the ring $R_1({\CC})$ is freely generated by 
two Eisenstein series $E_4$ and $E_6$ of weights $4$ and $6$. 
Deligne determined in \cite{Deligne1975} 
the ring $R_1({\ZZ})$ and the rings $R_1({\FF}_p)$. He
showed that
$$
R_1({\ZZ})= {\ZZ}[c_4,c_6,\Delta]
/(c_4^3-c_6^2-1728 \, \Delta) \, ,
$$
where  $\Delta$ is a cusp form of weight $12$  and $c_4$ and $c_6$ are of 
weight $4$ and $6$.
Reduction modulo $p$ gives a surjection of $R_1({\ZZ})$ to $R_1({\FF}_p)$ for $p\geq 5$.
Moreover, Deligne showed that $R_1({\FF}_p)$ 
in characteristic $2$ and $3$ is given by
$$
R_1({\FF}_2)={\FF}_2[a_1,\Delta], \quad
R_1({\FF}_3)={\FF}_2[b_2,\Delta], \quad
$$
where in each case 
$\Delta$ is a cusp form of weight $12$ and $a_1$ (resp.\ $b_2$) is a modular
form of weight $1$ (resp.\ $2$).

In \cite{Igusa1960} Igusa determined the ring
$R_2({\CC})$. He showed that the subring $R_2^{\rm ev}({\CC})$ of even weight
modular forms is generated freely by modular forms of weight $4,6,10$ and $12$
and $R_2({\CC})$ is generated over $R^{\rm ev}_2({\CC})$ by a cusp form of weight
$35$ whose square lies in $R^{\rm ev}_2({\CC})$; see also \cite{Igusa1967}. 
Later (\cite{Igusa1979}) 
he also determined the ring $R_2({\ZZ})$; it has $15$ generators of weights
ranging from $4$ to $48$. 

In characteristic $p\geq 5$ the structure of the rings
$R_2({\FF}_p)$ is similar to that of $R_2({\CC})$, see \cite{N,B-N}; these are 
generated by forms of weight $4,6,10,12$ and $35$.
The structure of $R_2({\FF}_p)$ for $p=2$ and $3$ was determined recently in \cite{C-vdG3,vdG2020}.
All these cases can be dealt with easily using the approach with invariant
theory.
\bigskip

In degree $2$ one can provide the $R_2$-module 
$$
M=\oplus_{j,k} M_{j,k}(\Gamma_2) \qquad \text{\rm with $M_{j,k}(\Gamma_2)=
H^0(\tilde{\mathcal{A}}_2,{\rm Sym}^j({\EE})\otimes \det({\EE})^k)$}
$$
with the structure of a ring using the projection of ${\rm GL}_2$-representations
${\rm Sym}^m(V) \otimes {\rm Sym}^n(V) \to {\rm Sym}^{m+n}(V)$ with $V$ the standard
representation by interpreting ${\rm Sym}^j(V)$ as the space of homogeneous polynomials
of degree $j$ in two variables, say $x_1,x_2$ and performing multiplication
of polynomials. The ring $M$ is not finitely generated as Grundh showed, 
see \cite[p.\ 234]{BGHZ}.
The dimensions of the spaces $S_{j,k}(\Gamma_2)({\CC})$ 
are known by Tsushima\cite{Tsushima} for $k\geq 4$;
for $k=3$ they were obtained independently by Petersen and Ta\"{\i}bi
\cite{Petersen, Taibi}.

For fixed $j$ the $R^{\rm ev}_2({\CC})$-modules 
$$
\oplus_{k} M_{j,2k}(\Gamma_2)({\CC}) \quad \text{\rm and} \quad
\oplus_{k} M_{j,2k+1}(\Gamma_2)({\CC})
$$ are finitely generated modules and their structure has 
been determined in a number of cases by Satoh, Ibukiyama and others, 
see the references in \cite{CFvdG2}. 
Invariant theory makes it easier to obtain such results.
\bigskip

For $g=3$ the results are less complete. 
Tsuyumine showed in 1985 (\cite{Tsuyumine}) that 
the ring $R_3({\CC})$ is generated by $34$ generators.
Recently Lercier and Ritzenthaler showed in \cite{LerRit} that 
$19$ generators suffice.

\end{section}
%%%%%%%%%%%%%%%%%%%%%%%%%%%%%%%%%%%%%%%%%
\begin{section}{Moduli of Curves of Genus two as a Stack Quotient}

We start with $g=2$. Let $F$ be a field of characteristic $\neq 2$ and $V=\langle
x_1,x_2 \rangle$ the $F$-vector space with basis $x_1,x_2$. The algebraic group 
${\rm GL}_2$ acts on $V$ via $(x_1,x_2)\mapsto (ax_1+bx_2,cx_1+dx_2)$ for
$(a,b;c,d) \in {\rm GL}_2(F)$. We will write $V_{j,k}={\rm Sym}^j(V) \otimes \det(V)^{\otimes k}$
for $j\in {\ZZ}_{\geq 0}$ and $k\in {\ZZ}$. This is an irreducible representation
of ${\rm GL}_2$. The underlying vector space can be identified with the space
of homogeneous polynomials of degree $j$ in $x_1,x_2$. We will denote by $V^0_{j,k}$
the open subspace of polynomials with non-vanishing discriminant. 

The moduli space $\mathcal{M}_2$ of smooth projective curves of genus $2$ over $F$
allows a presentation as an algebraic stack
$$
\mathcal{M}_2 \langepijl{\sim} [V^0_{6,-2}/{\rm GL}_2]
$$
Here the action of $(a,b;c,d) \in {\rm GL}_2(F)$ is by $f(x_1,x_2)\mapsto 
(ad-bc)^{-2}f(ax_1+bx_2,cx_1+dx_2)$.

Indeed, if $C$ is a curve of genus $2$ the choice of a basis $\omega_1,\omega_2$
of $H^0(C,K)$ with $K=\Omega_C^1$ defines a canonical map $C \to {\PP}^1$. Let $\iota$ denote the hyperelliptic
involution of $C$. Choosing a non-zero element $\eta \in H^0(C,K^3)^{\iota=-1}$
yields eight elements $\eta^2$, $\omega_1^6,\omega_1^5\omega_2,\ldots,\omega_2^6$
in the $7$-dimensional space $H^0(C,K^6)^{\iota=1}$ and thus a non-trivial
relation.

In inhomogeneous terms this gives us an equation $y^2=f$ with $f\in F[x]$ of degree $6$ with
non-vanishing discriminant. 
The space $H^0(C,K)$ has a basis $xdx/y, dx/y$.
%% on which
%%${\rm GL}_2$ should act by the standard representation. 
%%This necessitates the action $y \mapsto y(ad-bc)/(cx+d)^{3}$.
If we let ${\rm GL}_2$ act on $(x,y)$
via $(x,y) \mapsto ((ax+b)/(cx+d),y(ad-bc)/(cx+d)^3)$ then this action preserves
the form of 
the equation $y^2=f$ if we take $f$ in $V_{6,-2}$.
Then $\lambda \, {\rm Id}_V$
acts via $\lambda^2$ on $V_{6,-2}$. 
Thus the stabilizer of a generic element $f$ is of order $2$.
Moreover $- {\rm Id}_V$ acts by $y\mapsto -y$ on $y$ and the action
of ${\rm GL}_2$ on the differentials is by the standard representation.

\begin{conclusion} 
The pull back of the Hodge bundle ${\EE}$ on $\mathcal{M}_2$
under  the composition  
$V_{6,-2}^0\to [V^0_{6,-2}/{\rm GL}_2] \langepijl{\sim} \mathcal{M}_2$ is the
equivariant bundle $V$.
\end{conclusion}

The moduli space $\overline{\mathcal{M}}_2$ can be  constructed from
the projectivized space ${\PP}(V_{6,-2})$  of binary sextics.
The discriminant defines a hypersurface ${\DD}$ whose singular locus
has codimension $1$ in ${\DD}$.
The locus of binary sectics with three coinciding roots forms an
irreducible component ${\DD}'$ of the singular locus.
To illustrate the relation between
${\PP}(V_{6,-2})$ at a general point of ${\DD}'$ and
$\overline{\mathcal{M}}_2$ at a point of the locus $\delta_1$ in
$\overline{\mathcal{M}}_2$
of stable curves whose Jacobian is a product of two 
elliptic curves, 
we reproduce the picture of \cite[p.\ 80]{D-S}. 

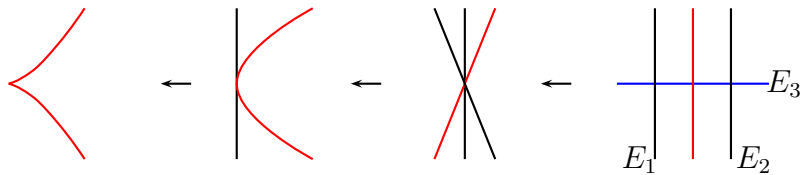
\begin{pspicture}(-2,-2)(10,3)
\psecurve[linecolor=red](-1,0)(0,0)(0.3,0.16)(0.5,0.35)(1,1)(2,2.82)
\psecurve[linecolor=red](-1,0)(0,0)(0.3,-0.16)(0.5,-0.35)(1,-1)(2,-2.82)
\psline{<-}(2,0)(2.4,0)
\psline{}(3,-1)(3,1)
\pscurve[linecolor=red](4,-1)(3,0)(4,1)
\psline{<-}(4.5,0)(4.9,0)
\psline(6,-1)(6,1)
\psline[linecolor=red](5.6,-1)(6.4,1)
\psline(6.4,-1)(5.6,1)
\psline{<-}(7,0)(7.4,0)
\psline[linecolor=blue](8,0)(10,0)
\psline[linecolor=red](9,-1)(9,1)
\psline(8.5,-1)(8.5,1)
\psline(9.5,-1)(9.5,1)
\rput(8.3,-1){$E_1$}
\rput(9.8,-1){$E_2$}
\rput(10.2,0){$E_3$}
\end{pspicture}  

Here we look at a plane $\Pi$
intersecting ${\DD}$ transversally at a general point of ${\DD}'$.
One blows up three times, starting at $\Pi \cap {\DD}'$, 
and then blows down the exceptional
divisors $E_1$ and $E_2$; after that
$E_3$ corresponds to the locus $\delta_1$ in $\overline{\mathcal M}_2$;
in $\mathcal{A}_2$ this corresponds to the locus $\mathcal{A}_{1,1}$
of product of elliptic curves.

\end{section}
%%%%%%%%%%%%%%%%%%%%%%%%%%%%%%%%%%%%%%%%
\begin{section}{Invariant Theory of Binary Sextics}
We review the invariant theory of ${\rm GL}_2$ acting on binary sextics.
Let $V=\langle x_1,x_2 \rangle$ be a $2$-dimensional vector space over a field $F$.
By definition an invariant for the action of ${\rm GL}_2$ acting on the space
${\rm Sym}^6(V)$ of binary sextics is an element invariant under ${\rm SL}_2(F)
\subset {\rm GL}_2(F)$. If we write
$$
f= \sum_{i=0}^6 a_i x_1^{6-i}x_2^i \eqno(2)
$$
for an element of ${\rm Sym}^6(V)$ and thus take $(a_0,\ldots,a_6)$ as coordinates
on ${\rm Sym}^6(V)$ then an invariant is a polynomial in $a_0,\ldots,a_6$
invariant under ${\rm SL}_2(F)$. The discriminant of a binary sextic, a polynomial of
degree $10$ in the $a_i$, is an example. 

For $F={\CC}$ the ring of invariants was determined
by Clebsch, Bolza and others in the 19th century. It is generated by invariants
$A,B,C,D,E$ of degrees $2,4,6,10$ and $15$ in the $a_i$.
Also  for $F={\FF}_p$ we have generators generators of these
degrees. We refer to \cite{Geyer,Igusa1960}. 

A covariant for the
action of ${\rm GL}_2$ on binary sextics is an element of $V \oplus {\rm Sym}^6(V)$
 invariant under the action of ${\rm SL}_2$. Such an element is a polynomial
in $a_0,\ldots,a_6$ and $x_1,x_2$. One way to make such covariants is to consider
equivariant embeddings of an irreducible ${\rm GL}_2$-representation $U$ into
${\rm Sym}^d({\rm Sym}^6(V))$. Equivalently, we consider an equivariant embedding
$$
\varphi: {\CC} \hookrightarrow {\rm Sym}^d({\rm Sym}^6(V)) \otimes U^{\vee} \, .
$$
Then $\Phi=\varphi(1)$ is a covariant. If $U$ has highest weight $(\lambda_1\geq \lambda_2)$
then $\Phi$ is homogeneous of degree $d$ in $a_0,\ldots,a_6$ and degree $\lambda_1-\lambda_2$
in $x_1,x_2$. We say that $\Phi$ has degree $d$ and order $\lambda_1-\lambda_2$.

The simplest example is the universal binary sextic $f$ given by (2); it corresponds to
taking $U={\rm Sym}^6(V)$. 

Another example is the Hessian of $f$. Indeed, we decompose
in irreducible representations
$$
{\rm Sym}^2({\rm Sym}^6(V))=V[12,0]\oplus V[10,2]\oplus V[8,4]\oplus V[6,6]\, ,
$$
where $V[a,b]={\rm Sym}^{a-b}(V) \otimes \det(V)^b$ is the irreducible representation
of highest weight $(a,b)$.
By taking $U=V[12,0]$ we find the covariant 
$\Phi=f^2$ and by taking $U=V[10,2]$ we get the Hessian;
$U=V[6,6]$ gives the invariant $A$.

The covariants form a ring $\mathcal{C}$ and the invariants form a subring $I=I(2,6)$.
The ring of covariants $\mathcal{C}$ was studied intensively at the end of the 19th century
and the beginning of the 20th century. The ring $\mathcal{C}$ is finitely generated
and Grace and Young presented $26$ generators for the ring 
$\mathcal{C}$, see \cite{G-Y}. 
These $26$ covariants are constructed as transvectants by differentiating
in a way similar to the construction of the Hessian. 
The $k$th transvectant of two forms $g\in {\rm Sym}^m(V)$,
$h\in {\rm Sym}^n(V)$ is defined as
$$
(g,h)_k=\frac{(m-k)!(n-k)!}{m!\,  n!}\sum_{j=0}^k (-1)^j
\binom{k}{j}
\frac{\partial^k g}{\partial x_1^{k-j}\partial x_2^j}
\frac{\partial^k h}{\partial x_1^{j}\partial x_2^{k-j}}
$$
and the index $k$ is usually omitted if $k=1$. Examples of the generators
are $C_{1,6}=f$, $C_{2,0}=(f,f)_6$, $C_{2,4}=(f,f)_4$, $C_{3,2}=(f,C_{2,4})_4$.
We refer to \cite{CFvdG2} for a
list of these $26$ generators.
\end{section}
%%%%%%%%%%%%%%%%%%%%%%%%%%%%%%%%%%%%%%%%
\begin{section}{Covariants of Binary Sextics and Modular Forms}
The Torelli morphism induces an embedding $\mathcal{M}_2 \hookrightarrow
\mathcal{A}_2$. The complement of the image is the locus $\mathcal{A}_{1,1}$
of products of elliptic curves. As a compactification we can take $\tilde{\mathcal{A}}_2=
\overline{\mathcal{M}}_2$. 

We now fix the field $F$ to be ${\CC}$ or a finite prime field ${\FF}_p$.

In the Chow ring ${\rm CH}_{\QQ}^*(\tilde{\mathcal A}_2)\otimes F$
we have the cycle relation
$$
10 \lambda_1= 2[\mathcal{A}_{1,1}]+[D]
$$
with $\lambda_1=c_1({\EE})$ the first Chern class of ${\EE}$
and $D$ the divisor that compactifies $\mathcal{A}_2 \otimes F$. This implies
that there exists a modular form of weight $10$ with divisor 
$2\, \mathcal{A}_{1,1}+D$, hence a cusp form.  
It is well-defined up to a non-zero multiplicative constant. We will 
normalize it later. We denote it by $\chi_{10}\in R_2(F)$. 

\smallskip
We let $V$ be the $F$-vector space with basis $x_1,x_2$.
The fact that the pullback of the Hodge bundle ${\EE}$ under
$$
V_{6,-2}^0 \to [V_{6,-2}^0/{\rm GL}_2] \to \mathcal{M}_2 \otimes F 
\hookrightarrow \mathcal{A}_2 \otimes F \eqno(3)
$$
is the equivariant bundle $V$ implies that a section of 
$L^k=\det({\EE})^k$ pulls back
to an invariant of degree $k$. We thus get an embedding of the ring of scalar-valued
modular forms of degree $2$ into the ring of invariants
$$
R_2(F) \hookrightarrow I(2,6)(F) \, .
$$
Conversely, an invariant of degree $d$ defines a section of $L^d$ on 
$\mathcal{M}_2 \otimes F$, hence
a rational (meromorphic) modular form of weight $d$ 
that is holomorphic outside $\mathcal{A}_{1,1}\otimes F$.
By multiplying it with an appropriate power of $\chi_{10}$ it becomes holomorphic on
$\mathcal{A}_2\otimes F$, hence on all of $\tilde{\mathcal{A}}_2 \otimes F$. 
We thus get maps
$$
R_2(F) \hookrightarrow I(2,6)(F) \langepijl{\nu} R_2(F)[1/\chi_{10}]
\eqno(4)
$$
the composition of which is the identity on $R_2(F)$.

From the description of the moduli $\overline{\mathcal{M}}_2$ given above 
one sees that the image of a cusp form is an invariant
divisible by the discriminant $D$. 
The image of $\chi_{10}$ is a non-zero multiple of the discriminant $D$. 
We may fix $\chi_{10}$ by requiring that $\nu(D)=\chi_{10}$.

\bigskip
This extends to the case of vector-valued
modular forms. Let 
$$
M(F)=\oplus_{j,k} M_{j,k}(\Gamma_2)(F)
$$
denote the ring of vector-valued modular forms of degree $2$. 
\begin{proposition}
Pullback via (3) defines homomorphisms
$$
M(F)  \hookrightarrow \mathcal{C}(2,6)(F) \langepijl{\nu} M(F)[1/\chi_{10}]\, ,
$$
the composition of which is the identity.
\end{proposition}
A modular form of weight $(j,k)$ corresponds to a covariant 
of degree $d=j/2+k$ and order $j$. A covariant of degree $d$ and order $r$
gives rise to a meromorphic modular form of weight $(r,d-r/2)$.

The most basic covariant is the universal binary sextic $f$.
By construction $\nu(f)$ is a meromorphic modular form of weight $(6,-2)$.
Therefore the central question is: {\it Which rational modular form is $\nu(f)$ ?}

\bigskip

Let $\mathcal{A}_{1,1} \subset \mathcal{A}_2$ be the locus of products of elliptic curves. Under the map
$$
\mathcal{A}_1 \times \mathcal{A}_1 \to \mathcal{A}_{1,1} \to \mathcal{A}_2
$$
the pullback of the Hodge bundle ${\EE}={\EE}^{(2)}$ is 
$p_1^* {\EE}^{(1)} \oplus p_2^*{\EE}^{(1)}$ with $p_1$ and $p_2$
the projections of $\mathcal{A}_1 \times \mathcal{A}_1$ on its factors.
The pullback of an element $h \in M_{j,k}(\Gamma_2)$ thus can be
indentified with an element of
$$
\bigoplus_{i=0}^j M_{k+j-i}(\Gamma_1)\otimes M_{k+i}(\Gamma_1) \, .
$$
Near a point of $\mathcal{A}_{1,1}$ 
we can write such an element symbolically as 
$$
h= \sum_{i=0}^j \eta_j X_1^{j-i}X_2^i\, ,
$$
where the $X_i$ are dummy variables to indicate the vector coordinates,
and such that the coefficient $\eta_j$
defines the element of $M_{k+j-i}(\Gamma_1)\otimes M_{k+i}(\Gamma_1)$.

In particular we have
$$
\nu(f)= \sum_{i=0}^6 \alpha_i X_1^{6-i}X_2^i\, ,
$$
where $\alpha_i$ are rational functions near a point of $\mathcal{A}_{1,1}$.
By interchanging $x_1$ and $x_2$
(that corresponds to the element
$\gamma \in \Gamma_2$ that interchanges $e_1$ and $e_2$)
we see that $\alpha_{6-i}=\alpha_i$ for $i=0,\ldots,3$.
%%%%%%%%%%%%%%%%%%%%%%%%%%
\begin{proposition}\label{prop68}
If ${\rm char}(F)\neq 2$ and $\neq 3$, then $\dim S_{6,8}(\Gamma_2)(F)=1$
and $\chi_{10} \nu(f)$ is a generator of $S_{6,8}(\Gamma_2)(F)$.
\end{proposition}
\begin{proof} We shall use that $\dim S_{6,8}(\Gamma_2)({\CC})\geq 1$.
Indeed, we know an explicit cusp form of weight $(6,8)$, see below.
(Alternatively, we know the dimensions of 
$S_{j,k}(\Gamma_2)({\CC})$ 
for $k\geq 4$, see \cite{Tsushima}; in particular we know 
$\dim S_{6,8}(\Gamma_2)({\CC})=1$.) 
By semi-continuity this implies that $\dim S_{6,8}(\Gamma_2)(F) \geq 1$.

The restriction of an element of $S_{6,8}(\Gamma_2)(F)$  
to the locus $\mathcal{A}_{1,1}\otimes F$ lands in
$$
\bigoplus_{i=0}^6 S_{8+6-i}(\Gamma_1)(F) \otimes S_{8+i}(\Gamma_1)(F) \, ,
$$
and as we have $\dim S_k(\Gamma_1)(F)=0$ for $k<12$ 
it vanishes on $\mathcal{A}_{1,1} \otimes F$.

The tangent space to $\mathcal{A}_2$ at a point $[X=X_1\times X_2]$
of $\mathcal{A}_{1,1}$, with $X_i$ elliptic curves, can be identified with 
$$
{\rm Sym}^2(T_X)= {\rm Sym}^2(T_{X_1}) \oplus (T_{X_1}\otimes T_{X_2}) 
\oplus {\rm Sym}^2(T_{X_2})
$$
with $T_X$ (resp $T_{X_i}$) the tangent space at the origin of $X$ (resp.\ $X_i$),
and with the middle term corresponding to the normal space.
Thus we see that the pullback of the conormal bundle of $\mathcal{A}_{1,1}$ to
$\mathcal{A}_1 \times \mathcal{A}_1$ is the tensor product of the
pullback of the Hodge bundles on the two factors $\mathcal{A}_1$. 

Let $h\in S_{6,8}(\Gamma_2)(F)$ and write $h$ as
$$
h=\sum_{i=0}^6 \eta_i \, X_1^{6-i}X_2^i
$$
locally at a general point of $\mathcal{A}_{1,1}\otimes F$.
If we consider the Taylor development in the normal direction of
$\mathcal{A}_{1,1}$ of the form $h$ that
vanishes on $\mathcal{A}_{1,1} \otimes F$ then the first non-zero 
Taylor term of $\eta_i$, say the $r$th term,  is an element of 
$$
S_{14-i+r}(\Gamma_1)(F) \otimes S_{8+i+r}(\Gamma_1)(F) \, .
$$
Since $S_k(\Gamma_1)(F)=(0)$ for $k<12$, 
a non-zero $r$th Taylor term of $\eta_i$ can occur only for  
$14-i+r\geq 12$ and $8+i+r\geq 12$. 
We thus find:
$$
{\rm ord}_{\mathcal{A}_{1,1}}(\eta_0,\ldots,\eta_6)
\geq (4,3,2,1,2,3,4)\, .
$$
\begin{lemma}
We have ${\rm ord}_{\mathcal{A}_{1,1}}(\eta_3)=1$.
\end{lemma}
\begin{proof} If ${\rm ord}_{\mathcal{A}_{1,1}}(\eta_3)\geq 2$ then
$h/\chi_{10}$ is a regular
form in $S_{6,-2}(\Gamma_2)$ and we write it as
$h/\chi_{10}=\sum_{i=0}^6 \xi_i\,  X_1^{6-i}X_2^i$
with $\xi_i=\eta_i/\chi_{10}$ regular. Then the invariant
$A=120 \, a_0a_6-20\, a_1a_5+8\, a_2a_4 -3 \, a_3^2$ defines 
a non-zero regular modular form 
$$
\nu(A)=120 \, \xi_0\xi_6 -20 \, \xi_1\xi_5 +8\, \xi_2\xi_4 -3\, \xi_3^2
$$
in $M_2(\Gamma_2)(F)$. 
But restriction to $\mathcal{A}_{1,1}$ gives for even $k$ an exact sequence
$$
0\to M_{k-10}(\Gamma_2)(F) \to M_k(\Gamma_2)(F) \to 
{\rm Sym}^2(M_k(\Gamma_1)(F)) \eqno(4)
$$ 
with the second arrow multiplication by $\chi_{10}$. This implies that
$\dim M_2(\Gamma_2)(F)=0$ for ${\rm char}(F)\neq 2$ and $\neq 3$.
This proves the lemma. \end{proof}

The image of a non-zero element $\chi_{6,8}$ of $S_{6,8}$  
in $\mathcal{C}(2,6)$ is 
a covariant of degree $11$ and order $6$. 
But since $\chi_{6,8}$ is a cusp form, 
this covariant is divisible
by the discriminant which is of degree $10$. 
Therefore, $\chi_{6,8}/\chi_{10}$
corresponds to a covariant of degree $1$, hence is a non-zero multiple of $f$.
This implies that $\dim S_{6,8}(\Gamma_2)(F)=1$. 
\end{proof}

\begin{corollary}\label{order}
If we write $\nu(f)=\sum_{i=0}^6 \alpha_i X_1^{6-i}X_2^i$ then
$$
{\rm ord}_{\mathcal{A}_{1,1}}(\alpha_0,\ldots,\alpha_6)\geq (2,1,0,-1,0,1,2)
$$
and ${\rm ord}_{\mathcal{A}_{1,1}}(\alpha_3)=-1$.
\end{corollary}
\end{section}
%%%%%%%%%%%%%%%%%%%
\begin{section}{Constructing vector-valued modular forms of degree $2$}
Now that we know $\nu(f)$  by Proposition \ref{prop68}
we can describe the map $\nu: \mathcal{C}(2,6) \to M[1/\chi_{10}]$ 
explicitly.
Recall that a covariant is a polynomial in $a_0,\ldots,a_6$ 
and $x_1,x_2$. We arrive at the following conclusion.

\begin{proposition}\label{substitution}
The map $\nu: \mathcal{C}(2,6)\to M[1/\chi_{10}]$
is substitution of $\alpha_i$ for $a_i$ (and $X_i$ for $x_i$).
\end{proposition}

In order to efficiently apply the proposition we
need to know the coordinates of a generator $\chi_{6,8}$ 
of $S_{6,8}$ very explicity. 

\begin{remark}
If $F\neq {\FF}_2$ the
 moduli space $\mathcal{A}_2[2]$ of level $2$ is a Galois cover of $\mathcal{A}_2$
with group ${\rm Sp}(2,{\ZZ}/2{\ZZ})$. This group is isomorphic to the symmetric group
$\mathfrak{S}_6$. The sign character of $\mathfrak{S}_6$ defines a character $\epsilon$
of $\Gamma_2$. The pullback of $\chi_{10}$ under
$\pi: \mathcal{A}_2[2]\to \mathcal{A}_2$ is a square $\chi_5^2$
since the pullback of $D$  under $\tilde{\mathcal{A}}_2[2]\to \tilde{\mathcal{A}}_2$
is divisible by $2$ as a divisor.
Thus  $\chi_5$ is a modular form of weight $5$ with character $\epsilon$.
\end{remark}

Let now $F={\CC}$. Recall that $\chi_{6,8}$ vanishes on $\mathcal{A}_{1,1}$. 
Dividing $\chi_{6,8}$  by $\chi_5$ provides a holomorphic 
vector-valued modular form $\chi_{6,3}
\in M_{6,3}(\Gamma_2,\epsilon)(F)$ with character $\epsilon$. 
Such a form can be constructed as follows.

We consider the six odd order two theta functions 
$\vartheta_i(\tau,z)$ with $(\tau,z)\in \mathfrak{H}_2\times {\CC}^2$. 
The gradient 
$G_i=(\partial \vartheta_i/\partial z_1, \partial \vartheta_i/\partial z_2)(\tau,0)$
is a modular form of weight $(1,1/2)$ on some congruence subgroup, but
the product of the transposes of these six gradients
defines a vector-valued modular form of weight $(6,3)$ on $\Gamma_2$ 
with character $\epsilon$.
The product $\chi_{6,8}=\chi_5 \chi_{6,3}$ is a cusp form of weight $(6,8)$ 
on $\Gamma_2$. A non-zero multiple of its
Fourier expansion starts with
(with $q_1=e^{2\pi i \tau_{11}}$,
$q_2=e^{2\pi i \tau_{22}}$ and $r=e^{2\pi i \tau_{12}}$)

\begin{align*}
\chi_{6,8}(\tau)=&
\left(
\begin{smallmatrix}
0\\
0\\
r^{-1}-2+r\\
2(r-r^{-1})\\
r^{-1}-2+r\\
0\\
0
\end{smallmatrix}
\right)
q_1q_2+
\left(
\begin{smallmatrix}
0\\
0\\
-2(r^{-2}+8r^{-1}-18+8r+r^2)\\
8(r^{-2}+4r^{-1}-4r-r^2)\\
-2(7r^{-2}-4r^{-1}-6-4r+7r^2)\\
12(r^{-2}-2r^{-1}+2r-r^2)\\
-4(r^{-2}-4r^{-1}+6-4r+r^2)
\end{smallmatrix}
\right)
q_1q_2^2\\
&+
\left(
\begin{smallmatrix}
-4(r^{-2}-4r^{-1}+6-4r+r^2)\\
12(r^{-2}-2r^{-1}+2r-r^2)\\
-2(7r^{-2}-4r^{-1}-6-4r+7r^2)\\
8(r^{-2}+4r^{-1}-4r-r^2)\\
-2(r^{-2}+8r^{-1}-18+8r+r^2)\\
0\\
0
\end{smallmatrix}
\right)
q_1^2q_2
+
\left(
\begin{smallmatrix}
16(r^{-3}-9r^{-1}+16-9r+r^3)\\
-72(r^{-3}-3r^{-1}+3r-r^3)\\
+128(r^{-3}-2+r^3)\\
-144(r^{-3}+5r^{-1}-5r-r^3)\\
+128(r^{-3}-2+r^3)\\
-72(r^{-3}-3r^{-1}+3r-r^3)\\
16(r^{-3}-9r^{-1}+16-9r+r^3)\\
\end{smallmatrix}
\right)
q_1^2q_2^2
+\dots\\
\end{align*}

Proposition \ref{substitution} provides an extremely effective way of constructing 
complex vector-valued Siegel modular forms of degree $2$. 
Let us give a few examples. In the decomposition
$$
{\rm Sym}^2({\rm Sym}^6(V))=V[12,0]\oplus V[10,2]\oplus V[8,4]\oplus V[6,6]
$$  
of ${\rm Sym}^2({\rm Sym}^6(V))$
the covariant $H$ defined by $V[10,2]$ is the Hessian and 
by Corollary \ref{order} gives rise to a form $\chi_{8,8}=
\nu(H)\chi_{10}
\in S_{8,8}(\Gamma_2)$ 
and using the Fourier expansion of $\chi_{6,8}$ we obtain the
Fourier expansion of $\chi_{8,8}$. 
Similarly, the covariant corresponding to $V[8,4]$
gives a form $\chi_{4,10}$ after multiplication with $\chi_{10}$. Finally, the covariant
defined by $V[6,6]$ is the invariant $A$ and defines the cusp form 
$\chi_{12}=\nu(A) \chi_{10}$. We refer to \cite{CFvdG} for more details.

As illustration of this we refer to the website 
\cite{website} that gives the Fourier series for  generators 
for all cases where $\dim S_{j,k}(\Gamma_2)=1$.

Another illustration of the
 efficacity of the construction of modular forms appears when one
considers the 
modules $\oplus_k M_{j,k}(\Gamma_2)$ and $\oplus_k M_{j,k}(\Gamma_2,\epsilon)$.
Let $R^{\rm ev}_2$ be the ring of scalar-valued modular forms of even weight.
The structure of the $R^{\rm ev}_2$-modules
$$
\oplus_k M_{j,2k}(\Gamma_2), \qquad \oplus_k M_{j,2k+1}(\Gamma_2)
$$
has been determined for $j=2,4,6,8,10$ by Satoh, 
Ibukiyama, Kiyuna, van Dorp and 
Takemori using various methods. 
Using covariants one can uniformly treat these cases and the cases of modular forms with character for the same values of $j$
$$
\oplus_k M_{j,2k}(\Gamma_2,\epsilon), \qquad \oplus_k M_{j,2k+1}(\Gamma_2,\epsilon)\, .
$$
For example, the $R^{\rm}_2$-module $\oplus_k M_{2,2k+1}(\Gamma_2,\epsilon)$ is free
with generators of weight $(2,9)$, $(2,11)$ and $(2,17)$ and the module $\oplus_k M_{10}^{10,2k}(\Gamma_2,\epsilon)$ is free with $10$ generators. 
We refer to \cite{CFvdG2}.

Yet another application of the construction of 
modular forms via covariants deals with
small weights. It is known by Skoruppa (\cite{Sk}) 
that $\dim S_{j,1}(\Gamma_2)=0$. He
proved this using Fourier-Jacobi forms. 
We conjecture $\dim S_{j,2}(\Gamma_2)=0$
and proved this for $j\leq 52$ using covariants.  
We refer to \cite{C-vdG2}

As a final illustration, 
for $k=3$ 
the smallest $j$ such that $\dim S_{j,3}(\Gamma_2)\neq 0$ is $36$. It is
not difficult to construct a generator of $S_{36,3}(\Gamma_2)$ using covariants, see \cite{C-vdG2}.
\end{section}
%%%%%%%%%%%%%%%%%%%
\begin{section}{Rings of Scalar-Valued Modular Forms}
The approach explained in the preceding section makes it easy
to find generators for the rings $R_2(F)=\oplus_k M_k(\Gamma_2)(F)$ of modular forms of degree $2$ for $F={\CC}$ or $F={\FF}_p$.
We write $\nu_F$ for the map $I(2,6)(F) \to R_2(F)[1/\chi_{10}]$. 
We denote by $R^{\rm ev}_2(F)$ the subring of even weight modular forms.

The degree $2$ invariant $A$ of a binary sextic
$f=\sum_{i=0}^6 a_i x_1^{6-i}x_2^i$
can be written as
$$
120\, a_0a_6-20\, a_1a_5+8\, a_2a_4 -3\, a_3^2\, .
$$
Corollary \ref{order}  implies that
$\nu_F(A)$ cannot be regular for $F={\CC}$ or ${\FF}_p$ with 
$p\geq 5$, but also that $\nu_F(AD)$
is a cusp form $\chi_{12} \in S_{12}(\Gamma_2)(F)$ of weight $12$.

In degree $4$ there is the invariant $B$ given by
$$
 (81 \, a_0a_6 +9 \, a_1a_5) a_3^2 
 -3\, (15\, a_0a_4a_5 +15\, a_1a_2a_6 
+a_1a_4^2+a_2^2a_5) a_3 +\cdots + a_2^2a_4^2 
$$
and Corollary \ref{order} implies
that it defines a regular modular form $\psi_4=\nu_F(B)$ 
of weight~$4$.

The invariant $C$ of degree $6$ is given by
$$
18\, (9\, a_0a_6 + 4\, a_1a_5)a_3^4- 6\, (33\, a_0a_4a_5+33\, a_1a_2a_6 +
4\, a_1a_4^2 +\, 4 a_2^2a_5)a_3^3 + \cdots
$$
and in a similar way one sees that $A B -3\, C$ starts with
$$
1458 \, a_0a_6 a_3^4 -486\, (a_0a_4a_5+a_1a_2a_6) a_3^3 + \cdots
$$
and defines a regular modular form $\psi_6=\nu_F(AB-3C)$ of weight $6$. 

The discriminant $D$ starts as
$$
729 \, a_0^2a_6^2\, a_3^6 -54 (9\, a_0^2a_4a_5a_6 -2\, a_0^2a_5^3+9\, a_0a_1a_2a_6^2 
- 2\, a_1^3a_6^2) a_3^5 +\cdots
$$
and is seen to have order $2$ along $\mathcal{A}_{1,1}$. It defines a cusp
form that is a non-zero multiple of $\chi_{10}$.

\begin{proposition}
For $F={\CC}$ or $F={\FF}_p$ with $p\geq 5$ the modular forms $\psi_4$, $\psi_6$,
$\chi_{10}$ and $\chi_{12}$ generate $R^{\rm ev}_2(F)$. 
\end{proposition}
\begin{proof} The algebraic independence of $A,B,C,D$ shows that
the generators are algebraically independent. 
Therefore $\psi_4,\psi_6,\chi_{10}, \chi_{12}$ 
generate a graded subring $T(F)\subseteq R^{\rm ev}_2(F)$ 
such that for even $k$ we have
$$
\dim T_k(F)= \frac{k^3}{17280} + O(k^2)\, .
$$
Now by Riemann-Roch we have for even $k$
$$
\dim M_k(\Gamma_2)(F)= \frac{c_1(L)^3}{3!} k^3 + O(k^2)
$$
since $c_1(L)^3=1/2880$, \cite[p.\ 72]{vdG}. 
Therefore there cannot be more generators.
Note that $4\cdot 6 \cdot 10 \cdot 12=2880$.
\end{proof}
\begin{remark}
Restriction to $\mathcal{A}_{1,1}$ 
shows that $\psi_4,\psi_6,\chi_{10},\chi_{12}$
generate $M_k(\Gamma_2)(F)$ for $k\leq 12$. Let $d(k)=\dim_F M_k(\Gamma_2)(F)$
and $t(k)=\dim_F T_k(F)$. Then $t(k)\leq d(k)$  and for even $k$ the exact
sequence (4) yields
$$
d(k)\leq d(k-10) + \frac{c(k)(c(k)+1)}{2}
$$
with $c(k)=\dim_F M_k(\Gamma_1)(F)$.
Now one easily sees $t(k)-t(k-10)= c(k)(c(k)+1)/2$. Thus 
if we assume $d(k-10)=t(k-10)$ we get 
$$
t(k) \leq d(k) \leq d(k-10)+ \frac{c(k)(c(k)+1)}{2}= t(k)
$$ 
and this provides via induction another proof that $\psi_4, \psi_6, \chi_{10}$ and $\chi_{12}$ generate $R^{\rm ev}_2(F)$ for $F={\CC}$ or ${\FF}_p$ with $p\geq 5$. 
\end{remark}

The odd degree invariant $E$ (of degree $15$) of binary sextics 
starts with
$$
-729(a_0^2a_5^3-a_1^3a_6)a_3^{10} + \ldots 
$$
and one checks that it has order $-3$ along $\mathcal{A}_{1,1}$. So
$\chi_{10}^2 \nu_F(E)$ defines a regular cusp form 
$\chi_{35} \in S_{35}(\Gamma_2)(F)$ with order $1$ along $\mathcal{A}_{1,1}$.

\smallskip

Let now ${\rm char}(F)\neq 2$.
The locus in $\mathcal{A}_2\otimes F$ of principally polarized abelian
surfaces $X$ with ${\rm Aut}(X)$ containing ${\ZZ}/2{\ZZ}\times {\ZZ}/2{\ZZ}$
consists of two irreducible divisors $H_1=\mathcal{A}_{1,1}$ and
$H_4$, the Humbert surface of degree $4$ of abelian surfaces isogenous
with a product by an isogeny of degree $4$. In terms of moduli of curves,
$H_4$ is the locus of curves that are double covers of elliptic curves.
We know that the cycle class of $H_1+H_4$ is $35 \lambda_1$ in
${\rm Pic}_{\QQ}(\mathcal{A}_2)$, see \cite[p.\ 218]{vdG1987}.

\begin{lemma} Suppose that ${\rm char}(F)\neq 2$.
A modular form $f \in M_k(\Gamma_2)(F)$ with $k$ odd vanishes on $H_1$
and $H_4$.
\end{lemma}
\begin{proof}
An abelian surface $[X] \in H_1$ or $[X] \in H_4$ possesses an involution
that acts by $-1$ on $H^0(X,\Omega^2_X)$.
\end{proof}
\begin{corollary}
The form $\chi_{35}$ as a section of $L^{\otimes 35}$ 
has as divisor $H_1+H_4+D$ with $D$ the divisor at infinity.
\end{corollary}

We can now easily derive the results of Igusa and Nagaoka
(see \cite{Igusa1960,N}, and also \cite{Ichikawa}).

\begin{theorem}\label{thmRing} Let $F={\CC}$ or $F={\FF}_p$ with $p\geq 5$. Then the 
ring $R_2({\FF}_p)$ is generated over $R_2^{\rm ev}({\FF}_p)
=F[\psi_4,\psi_6,\chi_{10},\chi_{12}]$ by the cusp form $\chi_{35}$
of weight $35$ with $\chi_{35}^2\in R^{\rm ev}_2(F)$.
\end{theorem}
\begin{proof}
Any odd weight modular form vanishes on $H_1$ and $H_4$, hence is
divisible by $\chi_{35}$. 
\end{proof}
\begin{remark}
The same argument proves Theorem \ref{thmRing} for any commutative
ring $F$ in which $6$ is invertible. It can also be used to obtain
Igusa's result on the ring $R_2({\ZZ})$.
\end{remark}
\bigskip

Now positive characteristic sometimes allows more modular forms than 
characteristic zero. We know that the locus in $\mathcal{A}_g \otimes {\FF}_p$
of abelian varieties of $p$-rank $<g$ has cycle class $(p-1)\lambda_1$, \cite{vdG,E-vdG}. 
This implies that there is a non-zero modular form of weight $p-1$ in characteristic $p$.
This modular form is called the Hasse invariant of degree $g$ and weight $p-1$.
The image of the Hasse invariant of degree $g$ under the Siegel operator is the
Hasse invariant of degree $g-1$. 

The Hasse invariants for degree $1$ and characteristic $2$ and $3$ appear as the
generators  $a_1$ and $b_2$ in 
$$
R_1({\FF}_2)={\FF}_2[a_1,\Delta], \quad
R_1({\FF}_3)={\FF}_2[b_2,\Delta] \, .
$$
\bigskip

The degree $2$ invariant $A$  of binary sextics
reduces to $a_1a_5-a_2a_4$ modulo $3$ and in view of Conclusion 
\ref{order}
defines a form $\nu_{\FF_3}(A)\in M_2(\Gamma_2)({\FF}_3)$ and
it must agree with the
Hasse invariant (up to a non-zero multiplicative scalar) as there is
only one invariant of degree $2$ (up to multiplicative scalars).
A careful analysis of the invariants in characteristic $3$ 
leads to the description of the ring $R_2({\FF}_3)$ 
given in \cite{vdG2020}. 

\begin{theorem}
The subring ${\mathcal R}_2^{\rm ev}({\FF}_3)$ of modular forms of even weight
 is generated by forms of weights $2,10,12,14$ and $36$ and has the form
$$
{\mathcal R}_2^{\rm ev}({\FF}_3)= {\FF}_3[\psi_2, \chi_{10}, \psi_{12}, \chi_{14},\chi_{36}]/J
$$
with $J$ the ideal generated by the relation
$ \psi_2^3\chi_{36}-\chi_{10}^3\psi_{12}-\psi_2^2\chi_{10}\chi_{14}^2+\chi_{14}^3$.
Moreover, ${\mathcal R}_2({\FF}_3)$ is generated over ${\mathcal R}_2^{\rm ev}$
by a form $\chi_{35}$ of weight $35$ whose square lies in $R^{\rm ev}_2({\FF}_3)$.
The ideal of cusp forms is generated by $\chi_{10}, \chi_{14},\chi_{35}, \chi_{36}$.
\end{theorem}

The case of characteristic $2$ was treated in joint work with Cl\'ery 
in \cite{C-vdG2}.
In the case of characteristic $2$ a curve of genus $2$ is not described by a binary
sextic. Instead we find an equation
$$
y^2+a\, y+ b=0
$$
with $a$ (resp. $b$) in $k[x]$ of degree $\leq 3$ (resp.\ $\leq 6$) and 
the hyperelliptic involution is $y\mapsto y+a$. It comes with a basis $xdx/a, dx/a$
of regular differentials. In this case we look at pairs $(a,b) \in 
V_{3,-1}\times V_{6,-2}$ with $V_{n,m}={\rm Sym}^n(V) \otimes \det(V)^m$.
Let $\mathcal{V}^0 \subset V_{3,-1}\times V_{6,-2}$ 
be the open subset defining smooth hyperelliptic curves.
Now we have an action of ${\rm GL}_2$ and an action of ${\rm Sym}^3(V)$ via
$$
(a,b) \mapsto (a, b+v^2+va)
$$
Together this defines a stack quotient
$$
[\mathcal{V}^0/{\rm GL}_2 \ltimes V_{3,-1}]
$$
Now by an invariant we mean a polynomial in the coefficients $a_0,\ldots,a_3$
and $b_0,\ldots,b_6$ that is invariant under ${\rm SL}(V) \ltimes {\rm Sym}^3(V)$.
Let $\mathcal{K}$ be the ring of invariants. A first example is the square root of the discriminant
of $a$:
$$
K_1=a_0a_3+a_1a_2 \, .
$$
As an analogue of (4) we now get homomorphisms
$$
R_2({\FF}_2) \hookrightarrow \mathcal{K} \langepijl{\nu} R_2({\FF}_2)[1/\chi_{10}]
$$
the composition of which is the identity.

In order to construct characteristic $2$ invariants one can 
still uses binary sextics as Igusa suggested in \cite{Igusa1960}.
Indeed, one lifts the curve given by $y^2+ay+b=0$ to the Witt ring, say defined by
$y^2+\tilde{a}y+\tilde{b}=0$ and takes an invariant of the binary sextic given by
$\tilde{a}^2+4\tilde{b}$, then divides these by the appropriate power of $2$ 
and reduces modulo $2$. 

For example, the degree $2$ invariant of binary sextics yields in this way an invariant
$K_2$ that equals $K_1^2$. A degree $4$ invariant yields an invariant $K_4$ that turns out
to be divisible by $K_1$. We thus find an invariant $K_3$ of degree $3$. 

The Hasse invariant $\psi_1$ must map to $K_1$. As in characteristic $3$ a careful 
analysis gives the orders of $a_i$ and $b_i$ along $\mathcal{A}_{1,1}$ and we
can deduce for an invariant $K$ the order of $\nu(K)$ 
along $\mathcal{A}_{1,1}$.  
The ring $R_2({\FF}_2)$ was described in \cite{C-vdG2}.

\begin{theorem}
The ring $\mathcal{R}_2({\FF}_2)$ is generated by modular forms of weights
$1$, $10$, $12$, $13$, $48$
satisfying one relation of weight $52$:
$$
\mathcal{R}_2({\FF}_2) = {\FF}_2[\psi_1,\chi_{10},\psi_{12},\chi_{13}, \chi_{48}]/(R)
$$
with $R=\chi_{13}^4 + \psi_1^3\chi_{10}\chi_{13}^3+
\psi_1^4 \chi_{48} +
\chi_{10}^4 \psi_{12} $.
The ideal of cusp forms is generated by $\chi_{10},\chi_{13}$ and $\chi_{48}$.
\end{theorem}
\end{section}
%%%%%%%%%%%%%%%%%%%%%%%%%%%%%%%%%%%%%%%%
\begin{section}{Moduli of Curves of Genus Three and Invariant Theory of 
Ternary Quartics}

Now we turn to genus $3$ treated in \cite{CFvdG3} 
and consider the moduli space $\mathcal{M}_3^{\rm nh}$
of non-hyperelliptic curves of genus $3$ over a field $F$. This is an open part
of the moduli space $\mathcal{M}_3$ with as complement the divisor $\mathcal{H}_3$
of hyperelliptic curves.
Let now $V=\langle x_0,x_1,x_2\rangle$ be the $3$-dimensional $F$-vector space with basis
$x_0,x_1,x_2$. We let $V_{4,0,-1}$ be the irreducible representation ${\rm Sym}^4(V) \otimes
\det(V)^{-1}$. The underlying space is the space of ternary quartics. It contains the open
subset $V_{4,0,-1}^0$ of ternary quartics with non-vanishing discriminant; that is, the
ternary quartics that define smooth plane quartic curves.

It is known that $\mathcal{M}^{\rm nh}_3$ has a description
as stack quotient
$$
\mathcal{M}_3^{\rm nh} \langepijl{\sim} [V_{4,0,-1}^0/{\rm GL}_3]
$$
Indeed, if $C$ is a non-hyperelliptic curve of genus $3$ then a choice of basis of
$H^0(C,K)$ defines an embedding of $C$ into ${\PP}^2$ and the image satisfies an
equation $f(x_0,x_1,x_2)=0$ with $f$ homogeneous of degree $4$.
In order that the action on the space of differentials
with basis
$$
x_i \, (x_0dx_1-x_1dx_0)/(\partial f/\partial x_2), \qquad i=0,1,2
$$
is the standard representation $V$ we need to twist ${\rm Sym}^4(V)$ by $\det(V)^{-1}$.
Then $\lambda {\rm Id} \in {\rm GL}_3(F)$ acts by $\lambda$ on $V_{4,0,-1}$ and we
arrive at the familiar stack quotient $[Q/{\rm PGL}_3]$ with $Q$ the space of 
smooth projective curves of degree $4$ in ${\PP}^2$ by first dividing by the
multiplicative group of multiples of the diagonal.

\begin{conclusion}\label{conclusion-deg3}
The pull back of the Hodge bundle ${\EE}$ on $\mathcal{M}_3^{\rm nh}$ under
$$
V_{4,0,-1}^0 \to [V_{4,0,-1}^0]/{\rm GL}_2] \langepijl{\sim} \mathcal{M}_3^{\rm nh}
$$
is the equivariant bundle $V$.
\end{conclusion}

Therefore we now look at the invariant theory of ${\rm GL}_3$ 
acting on ternary
quartics ${\rm Sym}^4(V)$ with $V=\langle x,y,z\rangle$ 
the standard representation of
${\rm GL}_3(V)$. We write the universal ternary quartic $f$ as
$$
f= a_0 x^4+ a_1 x^3y+ \cdots + a_{14} z^4
$$
in a lexicographic way. We fix coordinates for $\wedge^2V$
$$
\hat{x}=y\wedge z, \, \hat{y}= z \wedge x, \, 
\hat{z}=x \wedge y\, .
$$ 
Recall that an irreducible representation $\rho$ of ${\rm GL}_3$ 
is determined
by its highest weight $(\rho_1 \geq \rho_2 \geq \rho_3)$. This
representation appears in
$$
{\rm Sym}^{\rho_1-\rho_2}(V)\otimes
{\rm Sym}^{\rho_2-\rho_3}(\wedge^2 V) \otimes \det(V)^3
$$
An invariant for the action of ${\rm GL}_3$ on ${\rm Sym}^4(V)$ is
a polynomial in $a_0,\ldots,a_{14}$ invariant under ${\rm SL}_3$.
Instead of the notion of covariant we consider here the notion of 
a concomitant. A concomitant is a polynomial in $a_0,\ldots,a_{14}$
and in $x,y,z$ and $\hat{x},\hat{y},\hat{z}$ that is invariant under
the action of ${\rm SL}_3$. The most basic example is the universal
ternary quartic $f$. 

Concomitants can be obtained as follows. One takes an equivariant map
of ${\rm GL}_3$-representations
$$
U \hookrightarrow {\rm Sym}^d({\rm Sym}^4(V))
$$
or equivalently the equivariant embedding 
$$
\varphi: {\CC} \langepijl{} {\rm Sym}^d({\rm Sym}^4(V)) \otimes U^{\vee}
$$
Then $\Phi=\varphi(1)$ is a concomitant. If $U$ is an irreducible representation of highest weight $\rho_1 \geq \rho_2 \geq \rho_3$ then $\Phi$ is 
of degree $d$ in $a_0,\ldots,a_{14}$, of degree $\rho_1-\rho_2$ in 
$x,y,z$ and degree $\rho_2-\rho_3$ in $\hat{x},\hat{y},\hat{z}$.

The invariants form a ring $I(3,4)$ and the concomitants 
$\mathcal{C}(3,4)$ form a module over $I(3,4)$. For more on the ring
$I(3,4)$
see \cite{Dixmier}.
\end{section}
%%%%%%%%%%%%%%%%%%%%%%%%%%%%%%%%%%%%%%%%
\begin{section}{Concomitants of ternary quartics and modular forms of degree $3$}
The starting point for the construction of modular forms 
of degree $3$ is the Torelli morphism
$$
t: \mathcal{M}_3 \to \mathcal{A}_3
$$
defined by associating to a curve of genus $3$ its Jacobian. 
This is a morphism of Deligne-Mumford stacks of 
degree $2$ ramified along the hyperelliptic
locus $\mathcal{H}_3$. Indeed, every abelian variety has 
an automorphism of order $2$, 
but a generic curve of genus $3$ does not have
non-trivial automorphisms. Hyperelliptic curves have an automorphism
of order $2$ that induces $-1_{\rm Jac}$ on the Jacobian. 

There is a Siegel modular form $\chi_{18} \in S_{18}(\Gamma_3)$
constructed by Igusa \cite{Igusa1967}.
It is defined as the product of the $36$ even theta constants of order $2$.
The divisor of $\chi_{18}$ in the standard compactification
(defined by the second Voronoi fan) 
$\tilde{\mathcal{A}}_3$ is
$$
\mathcal{H}_3 +2D
$$
with $D$ the divisor at infinity. 

The pullback under the Torelli morphism
of the Hodge bundle ${\EE}$ on $\mathcal{A}_3$ 
is the Hodge bundle of $\mathcal{M}_3$. The Hodge bundle on $\mathcal{M}_3$
extends to the Hodge bundle over $\overline{\mathcal{M}}_3$, 
denoted again by ${\EE}$. For each irreducible representation $\rho$ of
${\rm GL}_3$ 
have a bundle ${\EE}_{\rho}$ on $\overline{\mathcal{M}}_3$ constructed by
applying a Schur functor.
We thus can consider
$$
T_{\rho}=H^0(\overline{\mathcal{M}}_3,{\EE}_{\rho})
$$
and elements of it are called Teichm\"uller modular forms of weight $\rho$
and genus (or degree) $3$. There is an involution $\iota$ acting on the
stack ${\mathcal{M}}_3$ associated to the double cover
$\mathcal{M}_3 \to \mathcal{A}_3$. If the characteristic is not $2$
we can thus split $T_{\rho}$ into $\pm 1$-eigenspaces under $\iota$
$$
T_{\rho}=T_{\rho}^{+} \oplus T_{\rho}^{-}\, .
$$
We can identify the invariants under $\iota$ with Siegel modular forms
$$
T_{\rho}^{+}= M_{\rho}(\Gamma_3)  \eqno(5)
$$
while the space $T_{\rho}^{-}$ consists of the genuine Teichm\"uller 
modular forms. 

The pullback of $\chi_{18}$ to $\mathcal{M}_3$ is a square
$\chi_9^2$ with $\chi_9$ a Teichm\"uller modular form of weight $9$
constructed by Ichikawa \cite{Ichikawa1,Ichikawa2}. 

Using the identification (5) We have
$$
\chi_9 \, T_{\rho}^{-} \subset  S_{\rho'}(\Gamma_3) 
\quad \text{\rm with $\rho'=\rho\otimes {\det}^9$.}
$$

We will now use invariant theory of ternary quartics 
Conclusion \ref{conclusion-deg3} implies that the pullback of a scalar-valued
Teichm\"uller modular form of weight $k$ is an invariant of weight $3k$
in $I(3,4)$. An invariant of degree $3d$ defines a 
meromorphic Teichm\"uller modular form of weight $d$ on
$\overline{\mathcal{M}}_3$ that becomes holomorphic after multiplication
by an appropriate power of $\chi_9$. Indeed, an invariant of degree
$3d$ is defined by an equivariant embedding $\det(V)^{4d} 
 \hookrightarrow
{\rm Sym}^{3d}({\rm Sym}^4(V))$ or taking care of the necessary 
twisting by
$$
\det(V)^{d}  \hookrightarrow
{\rm Sym}^{3d}({\rm Sym}^4(V)) \otimes \det(V)^{-3d} \, .
$$
We thus get
$$
T \langepijl{} I(3,4) \langepijl{} T[1/\chi_9]\, ,
$$
where the composition of the arrows is the identity.
In particular, the Teichm\"uller modular form $\chi_9$ maps
to an invariant of degree $27$ and since it is a cusp form
one can check that it must be divisible by the discriminant,
hence is a multiple of the discriminant.

We can extend this to vector-valued Teichm\"uller modular forms
$$
\Sigma \langepijl{} \mathcal{C}(3,4) \langepijl{\nu} \Sigma[1/\chi_9]
$$
with the $T$-module $\Sigma$ defined as
$$
\Sigma=\oplus_{\rho} T_{\rho}
$$
with $\rho$ running through the irreducible representations of ${\rm GL}_3$.

We can ask what the image $\nu(f)$ of the universal ternary quartic is.
By construction it is a meromorphic modular form of weight $(4,0,-1)$.
Here the weight refers to the irreducible representation
${\rm Sym}^4(V)\otimes \det(V)^{-1}$ of ${\rm GL}_3$.

We know that there exists a holomorphic modular cusp form 
$\chi_{4,0,8}$ of weight $(4,0,8)$, see \cite{C-vdG} and below. 

\begin{proposition}
Over ${\CC}$ the Siegel modular 
modular form  $\chi_9 \, \nu(f)$ is a generator of 
$S_{4,0,8}(\Gamma_3)({\CC})$.
\end{proposition}
\begin{proof}
The cusp form $\chi_{4,0,8}$ maps to a concomitant of degree $28$
that is divisible by the discriminant.
Therefore, $\chi_{4,0,-1}= \chi_{4,0,8}/\chi_9$
corresponds to a concomitant of degree $1$. This must be a non-zero multiple of $f$.
\end{proof}

If we write the universal ternary quartic lexicographically as
$$
f=a_0 x^4+ a_1x^3y+ \cdots + a_{14} z^4
$$
and we write the meromorphic Teichm\"uller form $\chi_{4,0,-1}$
similarly lexicographically as
$$
\chi_{4,0,-1}= \alpha_0 X^4+ \alpha_1 X^3Y + \cdots + \alpha_{14} Z^4$$
with dummy variables $X,Y,Z$ to indicate the coordinates of
$\chi_{4,0,-1}$, we arrive at the analogue for degree $3$:

\begin{proposition}
The map $\nu: \mathcal{C}(3,4) \to T[1/\chi_9]$ is given by
substituting $\alpha_i$ for $a_i$ (and $X,Y,Z$ for $x,y,z$ and
$\hat{X}, \hat{Y}, \hat{Z}$ for $\hat{x},\hat{y},\hat{z}$).
\end{proposition}
In the following we restrict to $F={\CC}$.
One way to construct a generator of $S_{4,0,8}(\Gamma_3)({\CC})$
 is to take the Schottky
form of degree $4$ and weight $8$ that vanishes on the
Torelli locus. We can develop it along
$\mathcal{A}_{3,1}$, the locus in $\mathcal{A}_4$ of products of
abelian threefolds and elliptic curves. It restriction to
$\mathcal{A}_{3,1}$ is a form in $S_8(\Gamma_3)\otimes S_8(\Gamma_1)$ and thus vanishes. 
The first non-zero term in the Taylor expansion along
$\mathcal{A}_{3,1}$ is
$$
\chi_{4,0,8}\otimes \Delta 
\in S_{4,0,8}(\Gamma_3)\otimes S_{12}(\Gamma_1)
$$
Since the Schottky form can be constructed explicitly with theta functions we can easily obtain the beginning of the Fourier expansion.
We refer to \cite{C-vdG} for the details.

In \cite{CFvdG3} we formulated a criterion that tells us which
elements of $\mathcal{C}(3,4)$ will give holomorphic modular
forms. We can associate to a  concomitant its order along
the locus of double conics by looking at its order in $t$
when we evaluate it on the ternary quartic $t\, f + q^2$ where
$q$ is a sufficiently general quadratic form in $x,y,z$. 
Then the result is
the following, see \cite{CFvdG3}.

\begin{theorem}
Let $c$ be a concomitant of degree $d$ and $v(c)$ its order
along the locus of double conics. 
If $d$ is odd then $\nu(c) \chi_9$ is a Siegel modular form 
with order $v(c)-(d-1)/2$ along the
hyperelliptic locus. If $d$ is even, then the order of 
$\nu(c)$ is $v(c)-d/2$.
\end{theorem}

We formulate a corollary.
Let $M_{i,j,k}(\Gamma_3)^{(m)}$ be the space of Siegel
modular forms of weight $(i,j,k)$ vanishing with multiplicity
$\geq m$ at infinity. (The weight $(i,j,k)$ corresponds to
the irreducible representation of ${\rm GL}_3$ 
of highest weight $(i+j+k,j+k,k)$.)
Moreover, let $\mathcal{C}_{d,\rho}(-m \, DC)$ be the
vector space of concomitants of type $(d,\rho)$ that have order
$\geq m$ along the locus of double conics. (Type $(d,\rho)$
means belonging to an irreducible representation $U$ of highest weight 
$\rho$ occurring in ${\rm Sym}^d({\rm Sym}^4(V))$.)

\begin{corollary}
The exists an isomorphism
$$
\mathcal{C}_{d,\rho}(-m \, DC) \langepijl{\sim}
M_{\rho_1-\rho_2,\rho_2-\rho_3,\rho_3+9(d-2m)}^{(d-2m)}
$$
given by $c \mapsto \nu(c) \chi_9^{d-2m}$.
\end{corollary}
This allows now the construction of Siegel modular forms
and Teichm\"uller modular forms of degree $3$. In fact, in principle, 
all of them. As a simple
example we decompose
$$
{\rm Sym}^2({\rm Sym}^4(V))= V[8,0,0]+V[6,2,0]+V[4,4,0]\, .
$$
The concomitant corresponding to $U=V[8,0,0]$ yields via $\nu$
the symmetric square of $\nu(f)$. The concomitant corresponding
to $V[6,2,0]$ yields a form in that after multiplication by
$\chi_{18}$ becomes a holomorphic form
in $S_{4,2,16}$ vanishing with multiplicity $2$ at infinity.
Similarly, the concomitant $c$ corresponding to 
 $U=V[4,4,0]$ yields a cusp form $\nu(c) \chi_{18} \in S_{0,4,16}$
vanishing with multiplicity $2$ at infinity. 
We refer for more examples to \cite{CFvdG3}.

The method also allows to treat the 
positive characteristic case. We hope to come back to it at another
occasion.  
\end{section}
%%%%%%%%%%%%%%%%%%%%%%%%%%%%%%%%%%%%%%%%
%%%%%%%%%%%%%%%%%%%%%%%%%%%%%%%%%%%%%%%%

\end{document}